\documentclass[final, 3p, times]{elsarticle}

\usepackage{amsmath,amsthm,amssymb}
\usepackage{bm}
\usepackage{multirow}

\usepackage[hidelinks]{hyperref}
\usepackage{cleveref}

\newcommand{\curl}{\nabla \times}
\renewcommand{\div}{\nabla \cdot}
\newcommand{\grad}{\nabla}

\newcommand{\enorm}[1]{{\left\vert\kern-0.1ex\left\vert\kern-0.1ex\left\vert #1 \right\vert\kern-0.1ex\right\vert\kern-0.1ex\right\vert}}

\newtheorem{theorem}{Theorem}
\newtheorem{lemma}[theorem]{Lemma}
\newtheorem{corollary}{Corollary}[theorem]

\numberwithin{equation}{section}
\numberwithin{theorem}{section}

\journal{}

\begin{document}
\begin{frontmatter}
\title{A multipoint vorticity mixed finite element method\\ for incompressible Stokes flow}
\author[polimi]{Wietse M. Boon\texorpdfstring{\corref{cor1}}{*}}
\author[polimi]{Alessio Fumagalli}
\cortext[cor1]{wietsemarijn.boon@polimi.it}
\address[polimi]{{Politecnico di Milano},
            {Piazza Leonardo da Vinci 32},
            {Milano},
            {Italy}}
\begin{abstract}
    We propose a mixed finite element method for Stokes flow with one degree of freedom per element and facet of simplicial grids. The method is derived by considering the vorticity-velocity-pressure formulation and eliminating the vorticity locally through the use of a quadrature rule. The discrete solution is pointwise divergence-free and the method is pressure robust. The theoretically derived convergence rates are confirmed by numerical experiments.
\end{abstract}

\end{frontmatter}

\section{Introduction}
\label{sec:introduction}

The velocity field that weakly solves a Stokes flow problem is typically sought in the Hilbert space $H^1$.
Stable, mixed finite element methods (MFEM) that satisfy this regularity constraint can be constructed through the use of discrete Stokes complexes \cite{falk2013stokes}.
However, we aim to construct a MFEM of low order using the Raviart-Thomas finite element pair, which lacks the regularity necessary to fit into the Stokes complex framework. Instead, we choose to conform to a three-field formulation of the Stokes equations, known as the \emph{vorticity-velocity-pressure} formulation \cite{dubois2003vorticity}. Recently, this formulation was discretized using the framework of finite element exterior calculus \cite{hanot2021arbitrary}.

The three-field formulation thus introduces the vorticity as a third variable and the main idea proposed in this letter is to subsequently eliminate this variable, locally, using a low-order quadrature rule. The technique is inspired by the \emph{multipoint flux MFEM} (MF-MFEM) \cite{wheeler2006multipoint}, and we inherit the nomenclature by referring to our method as a \emph{multipoint vorticity MFEM} (MV-MFEM). The method is linearly convergent with the velocity in $H(\div, \Omega)$ and the (locally post-computed) vorticity in $H(\curl, \Omega)$.

We remark that multipoint stress MFEM (MS-MFEM)\cite{caucao2020multipoint} was recently proposed for Stokes flow, based on the stress-velocity-vorticity formulation. In that case, the stress and vorticity are eliminated, which leads to a method that preserves local momentum balance but lacks local mass conservation, in contrast to MV-MFEM. The hybridization techniques in this family of methods were recently identified and generalized as local approximations of the exterior coderivative \cite{lee2018local}.

In this letter, we combine a low-order discretization of the three-field formulation from \cite{hanot2021arbitrary} with the computation of local coderivatives from \cite{lee2018local}. Focusing on the lowest order method, we show that the use of the localized quadrature rule does not impact the linear convergence. Moreover, the pressure variable is unaffected, as is the curl of the vorticity and the divergence of the velocity. In fact, in two dimensions, the only influence of the quadrature rule is a second-order error in the velocity. These results are shown theoretically in \Cref{sec:analysis} and experimentally in \Cref{sec:numerical_results}.

\subsection{Notation}

Let $L^2(\Omega)$ denote the space of square-integrable functions on $\Omega$ and let $H(\div,\Omega)$ and $H(\curl,\Omega)$ be the spaces of square-integrable vector fields on $\Omega$ with square-integrable divergence and curl, respectively.
The $H^s(\Sigma)$ norm is denoted $\| \cdot \|_{s, \Sigma}$ and we use the short-hand notation $\| \cdot \|_\Sigma := \| \cdot \|_{0, \Sigma}$ and $\| \cdot \| := \| \cdot \|_{0, \Omega}$.
We use $\langle \cdot, \cdot \rangle_\Omega$ to represent the $L^2$-inner product on $\Omega$ for scalar and vector-valued functions.
Angled brackets $\langle \cdot, \cdot \rangle$ denote duality pairings, $X^*$ is the dual of a Hilbert space $X$, and $B'$ is the adjoint of an operator $B$.
Parentheses $(\cdot, \cdot)$ are used to represent tuples.
In 2D, the curl operator is defined as $\curl r = [\partial_2, \ - \partial_1]^T r$ for scalar fields $r$ and as $\nabla \times q := \partial_1 q_2 -\partial_2 q_1$ for vector-fields $q$. The cross product with a vector $\nu$ is defined analogously.
The notation $a \lesssim b$, respectively $a \gtrsim b$, implies that a bounded constant $C > 0$ exists, independent of the mesh size, such that $Ca \le b$, respectively $a \ge Cb$. Moreover, we denote $a \eqsim b$ if $a \lesssim b$ and $b \lesssim a$.

\section{The continuous problem}
\label{sec:governing_equations}

Let $\Omega \subset \mathbb{R}^n$ be a simply connected, Lipschitz domain with $n \in \{2, 3\}$. Let $\mu > 0$ be the constant viscosity and $g \in (L^2(\Omega))^n$ a given body force. The governing equations for incompressible Stokes flow are given by momentum and mass balance for the velocity $q$ and pressure $p$, and we supplement these with the definition of the vorticity $r$:
\begin{subequations} \label{eqs: strong form}
\begin{align}
    - \div (2 \mu \varepsilon q - pI) &= g, &
    \div q &= 0, &
    r &= \mu \curl q.
\end{align}
Here, $\varepsilon$ is the symmetric gradient and $I \in \mathbb{R}^{n \times n}$ the identity tensor. We consider the following boundary conditions:
\begin{align}
    \nu \times q &= \nu \times q_0, &
    p &= p_0, &
    \text{on } &\partial_p \Omega, &
    \nu \times r &= 0, &
    \nu \cdot q &= 0, &
    \text{on } &\partial_q \Omega,
\end{align}
\end{subequations}
in which $\partial_p \Omega \cup \partial_q \Omega$ is a disjoint decomposition of the boundary $\partial \Omega$ and $p_0, q_0$ are given. In order to ensure uniqueness of $p$ and $q$, we assume that $\partial_p \Omega \ne \emptyset$ and $\| \nu \cdot \phi \|_{\partial_q \Omega} + \| \nu \times \phi \|_{\partial_p \Omega} > 0$ for all rigid body motions $\phi \ne 0$.

We continue by deriving the variational formulation of \eqref{eqs: strong form} in terms of $(r, q, p)$. First, we follow \cite{hanot2021arbitrary} and note the following calculus identity for solenoidal fields $q$:
\begin{align*}
    - \div ( \varepsilon q )
    = \frac12 \curl (\curl q)
    - \grad (\div q)
    = \frac12 \curl (\curl q).
\end{align*}
Using the definition of the vorticity $r := \mu \curl q$, this identity allows us to rewrite the momentum balance equation as
\begin{align*}
    \curl r + \div (pI) &= g.
\end{align*}

Second, let us define the composite Hilbert space $X := R \times Q \times P$ with
\begin{align*}
    R &:= \{ r \in H(\curl, \Omega) \mid \nu \times r|_{\partial_q \Omega} = 0 \}, &
    Q &:= \{ q \in H(\div, \Omega) \mid \nu \cdot q|_{\partial_q \Omega} = 0 \}, &
    P &:= L^2(\Omega).
\end{align*}

We are now ready to introduce the test functions $(\tilde r, \tilde q, \tilde p) \in X$ and apply integration by parts, which leads us to the variational, three-field formulation of problem~\eqref{eqs: strong form}:
Find $(r, q, p) \in X$ such that
\begin{subequations} \label{eqs: weak form}
\begin{align}
    \langle \mu^{-1} r, \tilde r \rangle_\Omega
    - \langle q, \curl \tilde r \rangle_\Omega
    \phantom{\ - \langle p, \div \tilde q \rangle_\Omega}
    &= \langle q_0, \nu \times \tilde r \rangle_{\partial_p \Omega}, &
    \forall \tilde r &\in R, \\
    \langle \curl r, \tilde q \rangle_\Omega
    \phantom{\ - \langle q, \curl \tilde r \rangle_\Omega}
    - \langle p, \div \tilde q \rangle_\Omega
    &= \langle g, \tilde q \rangle_\Omega
    - \langle p_0, \nu \cdot \tilde q \rangle_{\partial_p \Omega}, &
    \forall \tilde q &\in Q, \\
    \langle \div q, \tilde p \rangle_\Omega
    \phantom{\ - \langle p, \div \tilde q \rangle_\Omega}
    &= 0, &
    \forall \tilde p &\in P.
\end{align}
\end{subequations}

To highlight the structure of this problem, we define the operators $A, B_r, B_q$, and functionals $f_r, f_q$ such that:
\begin{align*}
    \langle A r, \tilde r \rangle &:= \langle \mu^{-1} r, \tilde r \rangle_\Omega, &
    \langle B_r r, \tilde q \rangle &:= \langle \curl r, \tilde q \rangle_\Omega, &
    \langle B_q q, \tilde p \rangle &:= \langle \div q, \tilde p \rangle_\Omega, \\
    \langle f_r, \tilde r \rangle &:= - \langle q_0, \tilde r \rangle_{\partial_q \Omega}, &
    \langle f_q, \tilde q \rangle &:=
    \langle g, \tilde q \rangle_\Omega
    - \langle p_0, \nu \cdot \tilde q \rangle_{\partial_q \Omega},
\end{align*}
for all $(r, q, p), (\tilde r, \tilde q, \tilde p) \in X$. Recall that unsubscripted, angled brackets $\langle \cdot, \cdot \rangle$ denote duality pairings. We then identify problem~\eqref{eqs: weak form} to be of the form
$\mathcal{B}(r,q,p) = f$ with $\mathcal{B}: X \to X^*$ and $f \in X^*$ given by
\begin{align*}
    \mathcal{B} &:=
    \begin{bmatrix}
        A & -B_r' \\
        B_r & & -B_q' \\
        & B_q \\
    \end{bmatrix}, &
    f &:=
    \begin{bmatrix}
        f_r \\ f_q \\ 0
    \end{bmatrix}.
\end{align*}

For the analysis of this problem, we define the energy norm as
\begin{align} \label{eq: enorm}
    \enorm{(r, q, p)}^2
    := \| r \|^2 + \| \curl r \|^2
    + \| q \|^2 + \| \div q \|^2
    + \| p \|^2.
\end{align}
It was shown in \cite[Lem.~3.4]{hanot2021arbitrary} that problem~\eqref{eqs: weak form} admits a unique solution that is bounded in the energy norm \eqref{eq: enorm}.

\section{Mixed finite element discretization}
\label{sec:discretization}

\subsection{A three-field mixed finite element method}

Let $\Omega_h$ be a shape-regular, simplicial tesselation of $\Omega$. For $n = 3$, the finite element spaces are chosen as the linear N\'ed\'elec elements of the second kind, the Raviart-Thomas space of lowest order, and the piecewise constants:
\begin{align*} 
    R_h &:= \mathbb{N}_1 \cap R, &
    Q_h &:= \mathbb{RT}_0 \cap Q, &
    P_h &:= \mathbb{P}_0 \cap P.
\end{align*}
We clarify that this choice of spaces leads to one degree of freedom per face for the velocity, one degree of freedom per element for the pressure, but two degrees of freedom per edge for the vorticity. This choice facilitates the hybridization introduced in the next section.
For $n=2$, we only adapt the vorticity space to the linear Lagrange elements, $R_h := \mathbb{L}_1 \cap R$.
It is important to note that $\curl R_h \subseteq Q_h$ and $\div Q_h = P_h$.
Finally, we define the composite space as $X_h := R_h \times Q_h \times P_h$.

The \emph{three-field mixed finite element method} (3F-MFEM) of problem~\eqref{eqs: weak form} is: Find $\hat x \in X_h$ such that
\begin{align} \label{eqs: discrete full}
    \langle \mathcal{B} \hat x, \tilde x \rangle
    &= \langle f, \tilde x \rangle, &
    \forall \tilde x \in X_h.
\end{align}

\begin{lemma}[{\cite[Thm.~3.7]{hanot2021arbitrary}}] \label{lem: conv of 3field}
    The 3F-MFEM is stable and convergent with the error bounded by
    \begin{align}
        \enorm{\hat x - x}
        \lesssim
        \inf_{\tilde x \in X_h}
        \enorm{\tilde x - x}.
    \end{align}
\end{lemma}

\subsection{A quadrature rule for the vorticity}
\label{sub:hybridization_of_the_vorticity}

Following \cite[Eq.~(4.6)]{lee2018local}, we introduce a quadrature rule for the vorticity variable and define the associated norm:
\begin{align} \label{eq: quad rule}
    \langle r, \tilde r \rangle_h
    &:= \sum_{\omega \in \Omega_h} \frac{|\omega|}{n + 1} \sum_{x \in \mathcal{V}(\omega)} (r_\omega \cdot \tilde r_\omega)(x), &
    \| r \|_h^2
    &:= \langle r, r \rangle_h, &
    \forall r, \tilde r \in \mathbb{P}_1^{k_n}(\Omega_h) \supseteq R_h,
\end{align}
with $r_\omega$ the restriction of $r$ on the simplex $\omega$ and $\mathcal{V}(\omega)$ the set of vertices of $\omega$.
We use $\mathbb{P}_l(\Omega_h)$ to denote the space of discontinous, elementwise polynomials of order $l$ and the coefficient is given by $k_3 = 3$ and $k_2 = 1$.
Finally, let us define the auxiliary space $W_h := \mathbb{P}_0^{k_n}(\Omega_h)$ containing elementwise constant (vector) fields.

Two key properties of the quadrature rule are highlighted in the following lemma.
\begin{lemma}[{\cite[Thm.~4.1]{lee2018local}}] \label{lem: equivalence}
    The quadrature rule from \eqref{eq: quad rule} satisfies the following properties
    \begin{align} \label{eqs: quadrature props}
        \| r \|_h &\eqsim \| r \|, &
        \forall r &\in R_h, &
        \langle \tilde r, w \rangle_h
        &= \langle \tilde r, w \rangle, &
        \forall \tilde r &\in \mathbb{P}_1^{k_n}, w \in W_h.
    \end{align}
\end{lemma}

 Let $A_h: R_h \to R_h^*$ be the operator associated with the quadrature rule such that $\langle A_h r, \tilde r \rangle := \langle \mu^{-1} r, \tilde r \rangle_h$. Next, we replace $A$ by $A_h$ in the system which leads us to define a new operator $\mathcal{B}_h: X_h \to X_h^*$:
\begin{align} \label{eq: def B_h}
    \langle \mathcal{B}_h (r, q, p), (\tilde r, \tilde q, \tilde p) \rangle
    &:= \langle \mathcal{B} (r, q, p), (\tilde r, \tilde q, \tilde p) \rangle
    + \langle A_h r, \tilde r \rangle
    - \langle A r, \tilde r \rangle, &
    \forall (r, q, p), (\tilde r, \tilde q, \tilde p) & \in X_h.
\end{align}
With this substitution in place, we consider the augmented system: Find $x_h \in X_h$ such that
\begin{align} \label{eqs: discrete three-field}
    \langle \mathcal{B}_h x_h, \tilde x \rangle
    &= \langle f, \tilde x \rangle, &
    \forall \tilde x \in X_h.
\end{align}

Similar to \eqref{eq: enorm}, the natural energy norm for this problem is given by
\begin{align} \label{eq: enorm h}
    \enorm{(r, q, p)}_h^2
    := \| r \|_h^2 + \| \curl r \|^2
    + \| q \|^2 + \| \div q \|^2
    + \| p \|^2,
\end{align}
and \Cref{lem: equivalence} directly gives us the equivalence relation $\enorm{x} \eqsim \enorm{x}_h$ for all $x \in X_h$.

\begin{lemma}[Well-posedness] \label{lem: discrete well-posed}
    Problem~\eqref{eqs: discrete three-field} admits a unique and bounded solution. In particular, the operator $\mathcal{B}_h$ from \eqref{eq: def B_h} satisfies
    \begin{align}
        \sup_{\tilde x \in X_h} \frac{\langle \mathcal{B}_h x_h, \tilde x \rangle}{\enorm{\tilde x}_h}
        &\eqsim
        \enorm{x_h}_h, &
        \forall x_h &\in X_h.
    \end{align}
\end{lemma}
\begin{proof}
    For the continuity bound ``$\lesssim$'', we note that $\mathcal{B}$ is continuous in the energy norm $\enorm{\cdot}$, which is equivalent to $\enorm{\cdot}_h$. Moreover, $A_h$ is continuous in the norm $\| \cdot \|_h$ and so $\mathcal{B}_h$, given by \eqref{eq: def B_h}, is continuous in the norm $\enorm{\cdot}_h$. For the lower bound ``$\gtrsim$'', we follow the same steps as in the proof of \cite[Lem.~3.4]{hanot2021arbitrary} and use the equivalence relation from \Cref{lem: equivalence} where necessary. Finally, we follow \cite[Lem.~3.5]{hanot2021arbitrary} to obtain the final requirements to invoke the Babu\v{s}ka-Lax-Milgram theorem for existence of the unique and bounded solution.
\end{proof}

\subsection{A multipoint vorticity mixed finite element method}
\label{sub: MV-MFA}

The quadrature rule \eqref{eq: quad rule} is local in the sense that $\langle r_i, \tilde r_j \rangle_h$ is only non-zero for basis functions $r_i, \tilde r_j \in R_h$ that have a degree of freedom at the same vertex. $A_h$ can thus be inverted by solving local systems around the vertices which allows us to eliminate the vorticity $r_h$. After elimination, we obtain the \emph{multipoint vorticity mixed finite element method} (MV-MFEM): Find $(q_h, p_h) \in Q_h \times P_h$ such that
\begin{align} \label{eqs: MV-MFE}
    \begin{bmatrix}
        B_r A_h^{-1} B_r' & -B_q' \\
        B_q \\
    \end{bmatrix}
    \begin{bmatrix}
        q_h \\ p_h
    \end{bmatrix}
    &=
    \begin{bmatrix}
        f_q - B_r A_h^{-1} f_r \\ 0
    \end{bmatrix}.
\end{align}

\section{A priori analysis of MV-MFEM}
\label{sec:analysis}

For the analysis of the MV-MFEM \eqref{eqs: MV-MFE}, we consider its equivalent formulation \eqref{eqs: discrete three-field}. We start by introducing suitable interpolants.
Let $\Pi_R$ and $\Pi_Q$ be the canonical interpolants onto the finite element spaces $R_h, Q_h$, respectively, defined for sufficiently regular $r \in R$ and $q \in Q$. Similarly, let $\Pi_P$ and $\Pi_W$ be the $L^2$ projections onto $P_h$ and $W_h$, respectively. These interpolants satisfy $\langle \curl (I - \Pi_R) r, q \rangle_\Omega = 0$ and $\langle \div (I - \Pi_Q) q, p \rangle_\Omega = 0$ for all $(r, q, p) \in X_h$, and have the approximation properties
\begin{subequations} \label{eqs: interpolants}
\begin{align}
    \| (I - \Pi_R) r \| + \| \curl ((I - \Pi_R) r) \|
    &\lesssim h~(\| r \|_{1, \Omega} + \| \curl r \|_{1, \Omega}), &
    \| (I - \Pi_P) p \|
    &\lesssim h~\| p \|_{1, \Omega}, \\
    \| (I - \Pi_Q) q \| + \| \div ((I - \Pi_Q) q) \|
    &\lesssim h~(\| q \|_{1, \Omega} + \| \div q \|_{1, \Omega}), &
    \| (I - \Pi_W) r \|
    &\lesssim h~\| r \|_{1, \Omega}.
\end{align}
\end{subequations}
In turn, the composite interpolant $\Pi_X$ is given by $\Pi_X (r, q, p) := (\Pi_R r, \Pi_Q q, \Pi_P p)$ for sufficiently regular $(r, q, p) \in X$.

\begin{theorem}[Convergence]\label{teo:conv}
    If the continuous solution $x \in X$ to \eqref{eqs: weak form} is sufficiently regular, then the discrete solution $x_h \in X_h$ to \eqref{eqs: discrete three-field} converges linearly in the energy norm \eqref{eq: enorm}:
    \begin{align*}
        \enorm{x_h - x} \lesssim h.
    \end{align*}
\end{theorem}
\begin{proof}
    We closely follow \cite{lee2018local} and start, as in \cite[Eq.~(3.11)]{lee2018local}, by evaluating the difference between $x_h := (r_h, q_h, p_h) \in X_h$, the solution to \eqref{eqs: discrete three-field}, and $\hat x := (\hat r, \hat q, \hat p) \in X_h$, the solution to \eqref{eqs: discrete full}. Using the equivalences from \Cref{lem: equivalence} and \Cref{lem: discrete well-posed}, we derive
    \begin{subequations} \label{eqs: convergence proof}
    \begin{align}
        \enorm{\hat x - x_h}
        \eqsim
        \enorm{\hat x - x_h}_h
        &\eqsim
        \sup_{\tilde x \in X_h} \frac{\langle \mathcal{B}_h(\hat x - x_h), \tilde x \rangle}{\enorm{\tilde x}_h}
        =
        \sup_{\tilde x \in X_h} \frac{\langle \mathcal{B}_h \hat x, \tilde x \rangle - \langle f, \tilde x \rangle}{\enorm{\tilde x}_h}
        =
        \sup_{\tilde r \in R_h} \frac{\langle A_h \hat r, \tilde r \rangle - \langle A \hat r, \tilde r \rangle}{\enorm{(\tilde r, 0, 0)}_h}.
    \end{align}

    Next, we follow the steps from \cite[Thm.~3.2]{lee2018local} to further bound the final term. Scaling the numerator by $\mu$ for ease of presentation, we use \Cref{lem: equivalence} and a Cauchy-Schwarz inequality to obtain
    \begin{align}
        \mu \left( \langle A_h \hat r, \tilde r \rangle - \langle A \hat r, \tilde r \rangle \right)
        = \langle \hat r, \tilde r \rangle_h - \langle \hat r, \tilde r \rangle_\Omega
        &= \langle \hat r - \Pi_W r, \tilde r \rangle_h - \langle \hat r - \Pi_W r, \tilde r \rangle_\Omega
        \lesssim  \| \hat r - \Pi_W r \| \ \| \tilde r \|_h.
    \end{align}
    \end{subequations}

    Combining \eqref{eqs: convergence proof}, the triangle inequality leads us to
    \begin{align*}
        \enorm{\hat x - x_h}
        \lesssim \mu^{-1} \sup_{\tilde r \in R_h} \frac{\| \hat r - \Pi_W r \| \ \| \tilde r \|_h}{\enorm{(\tilde r, 0, 0)}_h}
        \lesssim \| \hat r - \Pi_W r \|
        \le \| \hat r  - r \| + \| r - \Pi_W r \|.
    \end{align*}

    Next, we use another triangle inequality and \Cref{lem: conv of 3field} to derive
    \begin{align*}
        \enorm{x_h - x}
        \le
        \enorm{x_h - \hat x}
        + \enorm{\hat x - x}
        \lesssim
        \| r - \Pi_W r \|
        + \inf_{\tilde x \in X_h}
        \enorm{x - \tilde x}
        \le
        \| r - \Pi_W r \|
        + \enorm{x - \Pi_X x}.
    \end{align*}
    Finally, the properties of the interpolants \eqref{eqs: interpolants} and the assumed regularity of $x$ provides the result.
\end{proof}

Several advantageous properties of 3F-MFEM are preserved despite the use of the quadrature rule and we summarize these in the following two lemmas.
\begin{lemma}[Solenoidal] \label{lem: Solenoidal}
    The discrete solution $q_h \in R_h$ is point-wise solenoidal.
\end{lemma}
\begin{proof}
    The solution satisfies $\langle \div q_h, \tilde p \rangle_\Omega = 0$, for all $\tilde p \in P_h$. Since $\nabla \cdot R_h \subseteq P_h$, the result follows.
\end{proof}

\begin{lemma}[Pressure-robust] \label{lem: Pressure-robust}
    If the force term $g$ is perturbed by $\delta g = \nabla \phi$ for some $\phi$, then the solution $x_h$ is only perturbed in the component $p_h$.
\end{lemma}
\begin{proof}
    The proof of \cite[Thm 5.3]{hanot2021arbitrary} is unaffected by the substitution of $A$ by $A_h$, so it directly applies.
\end{proof}

In addition to the divergence of the velocity, certain components of the solution remain unchanged after including the quadrature rule. We formally present these invariants in the following theorem.
\begin{theorem}[Invariants] \label{thm: invariants}
    The introduction of the quadrature rule $A_h$ does not influence the pressure variable, nor the curl of the vorticity, i.e. $p_h = \hat p$ and $\curl r_h = \curl \hat r$.
\end{theorem}
\begin{proof}
    First, since $\curl R_h \subseteq Q_h$, we can construct a function $\tilde q_r := \curl (\hat r - r_h)$. Setting $\tilde x := (0, \tilde q_r, 0)$, we consider the difference between systems \eqref{eqs: discrete full} and \eqref{eqs: discrete three-field} and use the fact that $\div \tilde q_r = 0$ to derive:
    \begin{align*}
        0 &= \langle \mathcal{B} \hat x - \mathcal{B}_h x_h, \tilde x \rangle
        = \langle \curl(\hat r - r_h), \tilde q_r \rangle_\Omega
        - \langle \hat p - p_h, \div \tilde q_r \rangle_\Omega
        = \| \curl(\hat r - r_h) \|^2.
    \end{align*}
    Secondly, since $\div Q_h = P_h$, we can construct a function $\tilde q_p$ such that $\div \tilde q_p = \hat p - p_h$. Setting now $\tilde x := (0, \tilde q_r, 0)$ and using $\curl(\hat r - r_h) = 0$, we obtain
    \begin{align*}
        0 &= \langle \mathcal{B} \hat x - \mathcal{B}_h x_h, \tilde x \rangle
        = 
        - \langle \hat p - p_h, \div \tilde q_p \rangle_\Omega
        = - \| \hat p - p_h \|^2. \qedhere
    \end{align*}
\end{proof}

\subsection{The two-dimensional case}
\label{sub:the_two_dimensional_case}

\begin{corollary} \label{cor: invariant vorticity}
        If $n = 2$, then the vorticity is unaffected by the quadrature rule $A_h$, i.e. $r_h = \hat r$.
\end{corollary}
\begin{proof}
    Since $\curl(\hat r - r_h) = 0$ by \Cref{thm: invariants} and the curl is a rotated gradient in 2D, it follows that $\hat r - r_h$ is constant. Computing the difference of the two systems and setting now the test function as $\tilde x := (\hat r - r_h, 0, 0)$, we derive
    \begin{align*}
        0 &= \langle \mathcal{B} \hat x - \mathcal{B}_h x_h, \tilde x \rangle
        = \langle A \hat r - A_h r_h, \hat r - r_h \rangle
        - \langle \hat q - q_h, \curl (\hat r - r_h) \rangle_\Omega
        = \langle A (\hat r - r_h), \hat r - r_h \rangle
        = \mu^{-1} \| \hat r - r_h \|^2,
    \end{align*}
    where we used that $\hat r - r_h \in \mathbb{R} \subset W_h$ and \Cref{lem: equivalence}.
\end{proof}

In 2D, the definition of the curl and the use of $R_h = \mathbb{L}_1$ provide the following properties of the interpolants for sufficiently regular $r \in R$:
\begin{align} \label{eqs: 2D interpolants}
    \| (I - \Pi_W) r \|
    &\lesssim h~\| \curl r \|, &
    \| (I - \Pi_R) r \|
    &\lesssim h^2 ~\| r \|_{2, \Omega}.
\end{align}

\begin{lemma} \label{lem: }
    If $n = 2$, then $x_h$ converges quadratically to $\hat x$. In particular, $\enorm{ \hat x - x_h } = \| \hat q - q_h \| \lesssim h^2$.
\end{lemma}
\begin{proof}
    The equality follows from \Cref{lem: Solenoidal}, \Cref{thm: invariants} and \Cref{cor: invariant vorticity}. For the final estimate, we note that $\hat q - q_h$ is solenoidal, which allows us to take $\tilde r_q$ such that $\curl \tilde r_q = \hat q - q_h$. Setting $\tilde x := (\tilde r_q, 0, 0)$, we derive
    \begin{subequations} \label{eqs: proof 2D conv}
    \begin{align}
            0 &= \langle \mathcal{B} \hat x - \mathcal{B}_h x_h, \tilde x \rangle
            = \langle A \hat r - A_h r_h, \tilde r_q \rangle
            - \langle \hat q - q_h, \curl \tilde r_q \rangle_\Omega
            = \langle (A - A_h) r_h, \tilde r_q \rangle
            - \| \hat q - q_h \|^2.
        \end{align}
        Hence, we continue by using \eqref{eqs: 2D interpolants} to obtain the following bound
        \begin{align}
            \langle (A - A_h) r_h, \tilde r_q \rangle
            &= \langle (A - A_h) (I - \Pi_W) r_h, (I - \Pi_W) \tilde r_q \rangle
            \lesssim \| (I - \Pi_W) r_h \| ~\| (I - \Pi_W) \tilde r_q \|
            \lesssim h^2 \| \curl r_h \| ~ \| \hat q - q_h \|.
        \end{align}
        \end{subequations}
        Finally, we combine \eqref{eqs: proof 2D conv} and use the stability of the 3F-MFEM \cite[Lem.~3.4]{hanot2021arbitrary} to bound $\| \curl r_h \|$.
\end{proof}

\begin{theorem}[Improved estimate] \label{thm: improved estimate}
    If $n = 2$ and $r$ is sufficiently regular, then $r_h$ converges as $\| r_h - r \| \lesssim h^2$.
\end{theorem}
\begin{proof}
    We first consider the error equations for the 3F-MFEM with test function $\tilde x = \hat x - \Pi_X x$. Due to the properties of the canonical interpolants, the off-diagonal components cancel and we obtain:
    \begin{align*}
        0 &=
        \langle \mathcal{B} (\hat x - x), \tilde x \rangle
        = \mu^{-1} \langle \hat r - r, \hat r - \Pi_R r \rangle_\Omega
        = \mu^{-1} \left(\|\hat r - r \|^2 + \langle \hat r - r, r - \Pi_R r \rangle_\Omega \right).
    \end{align*}
    Next, \Cref{cor: invariant vorticity} gives us that $r_h = \hat r$ and so we conclude, using \eqref{eqs: 2D interpolants}, that
    \begin{align*}
        \| r_h - r \|
        &=\|\hat r - r \|
        = \frac{- \langle \hat r - r, (I - \Pi_R) r \rangle_\Omega}{\|\hat r - r \|}
        \le \| (I - \Pi_R) r \|
        \lesssim h^2 \| r \|_2. \qedhere
    \end{align*}
\end{proof}

\section{Numerical results} 
\label{sec:numerical_results}

In this section, we perform a convergence study for $n=2$ and 3 to verify the theoretical results of \Cref{sec:analysis}. The results are obtained with the library PorePy \cite{Keilegavlen2020} and PyGeoN \cite{pygeon}.
Let $\Omega :=
(0, 1)^n$, with the known solutions
\begin{align*}
    q(x, y) &= \curl
    (x^2 y^2 (x - 1)^2 (y - 1)^2), &
    p(x, y)
    &= xy(1 - x)(1 - y), &
    n &= 2, \\
    q(x, y, z) &=
    \curl
    \begin{bmatrix}
        (1 - x) x (1 - y)^2 y^2 (1 - z)^2 z^2 & 0 & 0
    \end{bmatrix}^T, &
    p(x, y, z)
    &= xyz(1 - x)(1 - y)(1 - z), &
    n &= 3.
\end{align*}
We set boundary conditions $q_0 = 0$ and $p_0 = 0$ on $\partial \Omega = \partial_p \Omega$, let $\mu := 1$ and compute the force term as $g = \div (2 \varepsilon q - pI)$.

\Cref{tab:err} compares MV-MFEM with 3F-MFEM in terms of numbers of degrees of freedom and relative $L^2$-errors. For fairness of comparison in the case $n = 3$, we choose $R_h$ in 3F-MFEM as the lowest-order N\'ed\'elec elements of the first kind, that has one degree of freedom per mesh edge.
The numbers
of degrees of freedom in MV-MFEM are nevertheless smaller due to the complete elimination of the vorticity.
The error
rates are in agreement with \Cref{teo:conv}, exhibiting at least linear convergence of
all variables with respect to $h$.
As proven in \Cref{thm: invariants}, we
observe numerically that the pressure, and therewith $\text{Error}_p$, is unaffected by the quadrature rule.

For $n = 2$, \Cref{cor: invariant vorticity} is reflected in the fact that the columns $\text{Error}_r$ are identical. The observed quadratic convergence in $r$ was shown in \Cref{thm: improved estimate}.
Finally, superconvergence of the pressure in the cell centers was observed (not
reported).

\begin{table}[htb]
    \centering
    \footnotesize\setlength{\tabcolsep}{4pt}
    \begin{tabular}{cc|ccccccc|ccccccc}
        & & \multicolumn{7}{|c|}{3F-MFEM}
        & \multicolumn{7}{|c}{MV-MFEM}\\
        & $h$ & $N_{\text{dof}}$ & $\text{Error}_r$ & $\text{Rate}_r$ & $\text{Error}_q$ & $\text{Rate}_q$ & $\text{Error}_p$ & $\text{Rate}_p$ & $N_{\text{dof}}$ & $\text{Error}_r$ & $\text{Rate}_r$ & $\text{Error}_q$ & $\text{Rate}_q$ & $\text{Error}_p$ & $\text{Rate}_p$ \\
        \hline
        \multirow{5}{*}{\rotatebox[origin=c]{90}{$n = 2$}}
        & 6.42e-2 & 1.91e3 & 7.73e-3 & -    & 6.62e-2 & -    & 1.12e-1 & -    & 1.57e3   & 7.73e-3 & -    & 6.95e-2 & -    & 1.12e-1 & -    \\
        & 3.17e-2 & 7.33e3 & 1.94e-3 & 1.96 & 3.28e-2 & 0.99 & 5.67e-2 & 0.97 & 6.06e3   & 1.94e-3 & 1.96 & 3.33e-2 & 1.04 & 5.67e-2 & 0.97 \\
        & 1.57e-2 & 2.88e4 & 4.86e-4 & 1.96 & 1.65e-2 & 0.98 & 2.84e-2 & 0.98 & 2.39e4  & 4.86e-4 & 1.96 & 1.65e-2 & 0.99 & 2.84e-2 & 0.98 \\
        & 7.83e-3 & 1.14e5 & 1.22e-4 & 1.99 & 8.26e-3 & 0.99 & 1.42e-2 & 0.99 & 9.52e4  & 1.22e-4 & 1.99 & 8.27e-3 & 1.00  & 1.42e-2 & 0.99 \\
        & 3.91e-3 & 4.56e5 & 3.05e-5 & 2.00 & 4.13e-3 & 1.00 & 7.12e-3 & 1.00 & 3.80e5 & 3.05e-5 & 2.00    & 4.13e-3 & 1.00 & 7.12e-3 & 1.00 \\
        \hline
        \hline
        \multirow{5}{*}{\rotatebox[origin=c]{90}{$n = 3$}}
        & 1.81e-1& 1.59e4 & 2.30e-1 & -     & 1.43e-1   & -    & 3.27e-1 & -  & 1.10e4  & 1.42e-1 &   -  & 1.68e-1 & -    & 3.27e-1 & -    \\
        & 1.65e-1& 2.11e4 & 2.11e-1 & 0.90  & 1.32e-1 & 0.83  & 2.92e-1 & 1.19 & 1.46e4 & 1.25e-1 & 1.34 & 1.48e-1 & 1.31 & 2.92e-1 & 1.19 \\
        & 1.48e-1& 2.86e4 & 1.89e-1 & 1.03  & 1.17e-1 & 1.13  & 2.58e-1 & 1.16 & 1.99e4 & 1.08e-1 & 1.32 & 1.28e-1 & 1.35 & 2.58e-1 & 1.16 \\
        & 1.36e-1& 3.62e4 & 1.73e-1 & 1.03  & 1.07e-1 & 1.05  & 2.38e-1 & 0.90 & 2.52e4 & 1.00e-1 & 0.90 & 1.16e-1 & 1.15 & 2.38e-1 & 0.90 \\
        & 1.27e-1& 4.43e4 & 1.61e-1 & 1.00  & 1.01e-1 & 0.83  & 2.22e-1 & 1.06 & 3.09e4 & 9.28e-2 & 1.14 & 1.09e-1 & 0.97 & 2.22e-1 & 1.06
    \end{tabular}
    \caption{Relative $L^2(\Omega)$-errors and convergence rates for the vorticity $r$, velocity $q$, and pressure $p$. The presented results are in agreement with the developed theory and show the
potential of the proposed approach.}
    \label{tab:err}
\end{table}

\section*{Acknowledgments}
\label{sec:acknowledgments}

This project has received funding from the European Union's Horizon 2020 research and innovation programme under the Marie Skłodowska-Curie grant agreement No. 101031434 -- MiDiROM.

\bibliographystyle{elsarticle-num}
\bibliography{references}
\end{document}